\def\myfontsize{10pt}
\IfFileExists{equivrigid.conf}{\input{equivrigid.conf}}{}
\documentclass[\myfontsize]{amsart}
\usepackage{amsmath, amsfonts, amssymb}
\usepackage{stmaryrd}
\usepackage{accents}
\usepackage[all]{xy}
\usepackage{graphicx}
\usepackage{hhline}
\usepackage{tikz}
\usetikzlibrary{patterns,calc}
\usepackage[colorlinks=true, linkcolor=blue,
  citecolor=blue, urlcolor=blue]{hyperref}
\ifcsname mypreamble\endcsname\mypreamble\fi

\def\newthm#1#2{\newtheorem{#1}[dummy]{#2}%
  \expandafter\def\csname#2\endcsname##1{\hyperref[#1:##1]{#2~\ref*{#1:##1}}}}
\newthm{thm}{Theorem}
\newthm{lemma}{Lemma}
\newthm{prop}{Proposition}
\newthm{cor}{Corollary}
\newthm{conj}{Conjecture}
\newthm{cons}{Consequence}
\newthm{claim}{Claim}
\newthm{fact}{Fact}
\newthm{obs}{Observation}
\newthm{guess}{Guess}
\newtheorem*{mainthm}{Theorem}
\theoremstyle{definition}
\newthm{defn}{Definition}
\newthm{remark}{Remark}
\newthm{example}{Example}
\newthm{question}{Question}
\newthm{notation}{Notation}
\newthm{problem}{Problem}

\newcommand{\Section}[1]{\hyperref[sec:#1]{Section~\ref*{sec:#1}}}
\newcommand{\Table}[1]{\hyperref[tab:#1]{Table~\ref*{tab:#1}}}
\newcommand{\Figure}[1]{\hyperref[fig:#1]{Figure~\ref*{fig:#1}}}
\newcommand{\eqn}[1]{\hyperref[eqn:#1]{(\ref*{eqn:#1})}}

\DeclareMathOperator{\GL}{GL}
\DeclareMathOperator{\SL}{SL}

\DeclareMathOperator{\Gr}{Gr}
\DeclareMathOperator{\Fl}{Fl}
\DeclareMathOperator{\LG}{LG}

\DeclareMathOperator{\SG}{SG}

\DeclareMathOperator{\OG}{OG}

\DeclareMathOperator{\Sym}{Sym}
\DeclareMathOperator{\Span}{Span}
\DeclareMathOperator{\Spec}{Spec}
\DeclareMathOperator{\QH}{QH}
\DeclareMathOperator{\QK}{QK}

\DeclareMathOperator{\codim}{codim}

\DeclareMathOperator{\Pic}{Pic}

\DeclareMathOperator{\dist}{dist}

\newcommand{\comin}{\mathrm{comin}}

\newcommand{\ssm}{\smallsetminus}
\newcommand{\bG}{{\mathbb G}}
\newcommand{\bA}{{\mathbb A}}
\newcommand{\bP}{{\mathbb P}}
\newcommand{\bK}{{\mathbb K}}
\newcommand{\C}{{\mathbb C}}

\newcommand{\Q}{{\mathbb Q}}
\newcommand{\Z}{{\mathbb Z}}
\newcommand{\N}{{\mathbb N}}

\def\cL{{\mathcal L}}

\newcommand{\cO}{{\mathcal O}}

\newcommand{\fm}{{\mathfrak m}}

\newcommand{\ds}{\displaystyle}

\newcommand{\pt}{\mathrm{point}}

\newcommand{\al}{{\alpha}}
\newcommand{\be}{{\beta}}
\newcommand{\ga}{{\gamma}}

\newcommand{\ka}{{\kappa}}
\newcommand{\la}{{\lambda}}

\newcommand{\om}{{\omega}}

\newcommand{\ev}{\operatorname{ev}}

\newcommand{\wb}{\overline}
\newcommand{\ov}{\overline}

\newcommand{\ignore}[1]{}

\newcommand{\Mb}{\wb{\mathcal M}}

\newcommand{\noin}{\noindent}

\begin{document}

\title{Equivariant rigidity of Richardson varieties}

\date{April 24, 2025}

\author{Anders~S.~Buch}
\address{Department of Mathematics, Rutgers University, 110
  Frelinghuysen Road, Piscataway, NJ 08854, USA}
\email{asbuch@math.rutgers.edu}

\author{Pierre--Emmanuel Chaput}
\address{Domaine Scientifique Victor Grignard, 239, Boulevard des
  Aiguillettes, Universit{\'e} de Lorraine, B.P.  70239,
  F-54506 Vandoeuvre-l{\`e}s-Nancy Cedex, France}
\email{pierre-emmanuel.chaput@univ-lorraine.fr}

\author{Nicolas Perrin}
\address{Centre de Math\'ematiques Laurent Schwartz (CMLS), CNRS, \'Ecole
polytechnique, Institut Polytechnique de Paris, 91120 Palaiseau, France}
\email{nicolas.perrin.cmls@polytechnique.edu}



\thanks{Buch was partially supported by NSF grant DMS-2152316. Perrin was
partially supported by ANR project FanoHK, grant ANR-20-CE40-0023.}

\begin{abstract}
  We prove that Schubert and Richardson varieties in flag manifolds are uniquely
  determined by their equivariant cohomology classes, as well as a stronger
  result that replaces Schubert varieties with closures of Bialynicki-Birula
  cells under suitable conditions. This is used to prove a conjecture from
  \cite{buch.chaput.ea:seidel}, stating that any two-pointed curve neighborhood
  representing a quantum cohomology product with a Seidel class is a Schubert
  variety. We pose a stronger conjecture which implies a Seidel multiplication
  formula in equivariant quantum $K$-theory, and prove this conjecture for
  cominuscule flag varieties.
\end{abstract}

\maketitle


\section{Introduction}

A Schubert variety $\Omega$ in a flag manifold $X = G/P$ is called \emph{rigid}
if it is uniquely determined by its class $[\Omega]$ in the cohomology ring
$H^*(X)$. More precisely, if $Z \subset X$ is any irreducible closed subvariety
such that $[Z]$ is a multiple of $[\Omega]$ in $H^*(X)$, then $Z$ is a
$G$-translate of $\Omega$. This problem has been studied in numerous papers, see
e.g.\ \cite{hong:rigidity, hong:rigidity*1, coskun:rigid, robles.the:rigid,
coskun.robles:flexibility, coskun:rigidity, coskun:restriction*1,
hong.mok:schur, liu.sheshmani.ea:multi-rigidity} and the references therein. In
this paper we show that all Schubert varieties and Richardson varieties are
\emph{equivariantly rigid}. In other words, if $T \subset G$ is a maximal torus,
$\Omega \subset X$ is a $T$-stable Richardson variety, and $Z \subset X$ is a
(non-empty) $T$-stable closed subvariety such that the $T$-equivariant class
$[Z] \in H^*_T(X)$ is a multiple of $[\Omega]$, then $Z = \Omega$.

More generally, let $T$ be an algebraic torus over an algebraically closed
field, let $X$ be a non-singular projective $T$-variety, and let $\Omega \subset
X$ be a $T$-stable closed subvariety. Let $\Omega^T$ denote the set of $T$-fixed
points in $\Omega$. We will say that $\Omega$ is \emph{$T$-convex} if, for any
$T$-stable closed subvariety $Z \subset X$ satisfying $Z^T \subset \Omega$, we
have $Z \subset \Omega$. A fixed point $p \in X^T$ is called \emph{fully
definite} if all $T$-weights of the Zariski tangent space $T_p X$ belong to a
strict half-space of the character lattice of $T$. We show that if all $T$-fixed
points in $X$ are fully definite, then any irreducible $T$-convex subvariety of
$X$ is also $T$-equivariantly rigid. Here $\Omega$ is called $T$-equivariantly
rigid if $\Omega$ is determined by its class in the $T$-equivariant Chow
cohomology ring of $X$.

Let $\bG_m \subset T$ be a 1-parameter subgroup such that $X^T = X^{\bG_m}$, and
assume that this fixed point set is finite. The associated Bialynicki-Birula
decomposition of $X$ is given by $X = \bigcup_{p \in X^T} X_p^+$, where $X_p^+ =
\{ x \in X \mid \lim_{t \to 0} t.x = p \}$ is the Bialynicki-Birula cell of
points attracted to $p$ by the action of $t \in \bG_m$. This decomposition is
called a \emph{stratification} if each cell closure $\ov{X_p^+} \subset X$ is a
union of smaller cells. In this case we show that all cell closures are
$T$-convex. We arrive at the following result combining \Theorem{rigid} and
\Proposition{bb-fpi}.

\begin{mainthm}
  Let $X$ be a non-singular projective $T$-variety with finitely many $T$-fixed
  points, and choose $\bG_m \subset T$ such that $X^T = X^{\bG_m}$.\smallskip

  \noin{\rm(a)}\, Assume that the Bialynicki-Birula decomposition of $X$ is a
  stratification. Then each cell closure $\ov{X_p^+}$ is $T$-convex.\smallskip

  \noin{\rm(b)}\, Assume that all $T$-fixed points in $X$ are fully definite.
  Then any irreducible $T$-convex subvariety of $X$ is $T$-equivariantly rigid.
\end{mainthm}

This result applies to Schubert and Richardson varieties in flag varieties, as
well as positroid varieties in Grassmannians, so these subvarieties are both
$T$-convex and $T$-equivariantly rigid. However, projected Richardson varieties
do not in general enjoy these properties, see \Remark{projrich}. Our theorem
also covers a class of horospherical varieties, which includes all non-singular
horospherical varieties of Picard rank 1 \cite{pasquier:some}.

Our theorem has additional applications in quantum Schubert calculus. Let $X =
G/P$ be a complex flag manifold. A Schubert class $[X^w]$ is called a
\emph{Seidel class} if the Weyl group element $w$ is the minimal representative
of a point in some cominuscule flag variety $G/Q$. Multiplication by Seidel
classes in the quantum cohomology ring $\QH(X)$ is given by the identity $[X^w]
\star [X^u] = q^{d(w,u)} [X^{w u}]$, where $d(w,u)$ is the unique minimal degree
of a rational curve connecting the opposite Schubert varieties $X_{w_0 w}$ and
$X^u$ \cite{seidel:1, belkale:transformation, chaput.manivel.ea:affine}. This
implies that $[X^{w u}]$ is equal to the class of the curve neighborhood
$\Gamma_{d(w,u)}(X_{w_0 w}, X^u)$, defined as the union of all stable curves in
$X$ of degree $d(w,u)$ connecting $X_{w_0 w}$ to $X^u$. We conjectured in
\cite{buch.chaput.ea:seidel} that this curve neighborhood is in fact the
translated Schubert variety
\begin{equation}\label{eqn:intro_conj}%
  \Gamma_{d(w,u)}(X_{w_0 w}, X^u) \,=\, w^{-1}.X^{w u} \,.
\end{equation}
This has been proved in some cases when $X$ is cominuscule, in all cases when
$X$ is a flag variety of type A \cite{li.liu.ea:seidel, tarigradschi:curve}, and
for $X=\SG(2,2n)$ \cite{benedetti.perrin.ea:quantum}. Using that
$\Gamma_{d(w,u)}(X_{w_0 w},X^u)$ and $w^{-1}.X^{wu}$ define the same class in
$H^*_T(X)$ by an equivariant version of the Seidel multiplication formula from
\cite{chaput.manivel.ea:affine, chaput.perrin:affine}, the identity
\eqn{intro_conj} follows from our result that Schubert varieties are
equivariantly rigid.

In this paper we conjecture the more general identity
\begin{equation}\label{eqn:intro_extend}%
  \Gamma_{d(w,u)+e}(X_{w_0 w}, X^u) \,=\, \Gamma_e(w^{-1}(X^{w u})) \,,
\end{equation}
where the right hand side is the union of all stable curves of degree $e$ that
pass through $w^{-1}.X^{w u}$. This union is a Schubert variety
\cite{buch.chaput.ea:finiteness} whose Weyl group element was determined in
\cite{buch.mihalcea:curve}. Let $M_{d(w,u)+e}(X_{w_0 w},X^u)$ denote the moduli
space of 3-pointed stable maps to $X$ of degree $d(w,u)+e$ and genus zero, which
send the first two marked points to $X_{w_0 w}$ and $X^u$, respectively. We
further conjecture that the evaluation map $\ev_3 : M_{d(w,u)+e}(X_{w_0 w},X^u)
\to \Gamma_e(w^{-1}(X^{w u}))$ is cohomologically trivial. This conjecture
implies a Seidel multiplication formula in the equivariant quantum $K$-theory
ring $\QK_T(X)$. We prove this conjecture when $X$ is a cominuscule flag
variety, thereby obtaining an equivariant generalization of our Seidel
multiplication formula from \cite{buch.chaput.ea:seidel}. Based on suggestions
from Mihail Tarigradschi, we finally apply the methods of
\cite{tarigradschi:curve} to prove the identity \eqn{intro_extend} when $X =
\GL_n(\C)/P$ is any flag manifold of Lie type A.

Our paper is organized as follows. In \Section{actions} we recall some basic
facts and notation related to torus actions. In \Section{local} we show that if
all $T$-fixed points of $X$ are fully definite, then the fixed point set $Z^T$
of a $T$-stable subvariety $Z \subset X$ is determined by its equivariant class
$[Z] \in H^*_T(X)$. This is used in \Section{rigidity} to prove part (b) of the
above theorem. \Section{bbcells} proves part (a). \Section{schubert} interprets
our theorem for flag varieties, which is used in \Section{seidel} to prove the
conjecture about curve neighborhoods from \cite{buch.chaput.ea:seidel}.
\Section{qkseidel} discusses the more general conjecture as well as its
consequences in quantum $K$-theory. Finally, \Section{horospherical} interprets
our theorem for certain horospherical varieties.

We are particularly grateful to Mihail Tarigradschi, whose suggestions led to
our proof of \eqn{intro_extend} in type A. We also thank Allen Knutson for
suggesting the term \emph{$T$-convex}. We finally thank an anonymous referee for
a very careful reading of our manuscript and for several insightful suggestions.


\section{Torus actions}\label{sec:actions}

We work with varieties over a fixed algebraically closed field $\bK$. Varieties
are reduced but not necessarily irreducible. A point will always mean a closed
point. The multiplicative group of $\bK$ is denoted $\bG_m = \bK \ssm \{0\}$. An
(algebraic) torus is a group variety isomorphic to $(\bG_m)^r$ for some $r \in
\N$.

Let $T = (\bG_m)^r$ be an algebraic torus. Any rational representation $V$ of
$T$ is a direct sum $V = \bigoplus_\la V_\la$ of weight spaces $V_\la = \{v \in
V \mid t.v = \la(t) v ~\forall t \in T\}$ defined by characters $\la : T \to
\bG_m$. The \emph{weights} of $V$ are the characters $\la$ for which $V_\la \neq
0$. The group of all characters of $T$ is called the \emph{character lattice}
and is isomorphic to $\Z^r$. Given a $T$-variety $X$, we let $X^T \subset X$
denote the closed subvariety of $T$-fixed points. A subvariety $Z \subset X$ is
called \emph{$T$-stable} if $t.z \in Z$ for all $t \in T$ and $z \in Z$. In this
case $Z$ is itself a $T$-variety.

\begin{defn}\label{defn:extremal}%
  The $T$-fixed point $p \in X$ is \emph{non-degenerate} in $X$ if $T$ acts with
  non-zero weights on the Zariski tangent space $T_p X$. The point $p$ is
  \emph{fully definite} if all $T$-weights of $T_p X$ belong to a strict
  half-space of the character lattice of $T$.
\end{defn}

Equivalently, $p \in X^T$ is fully definite in $X$ if and only if there exists a
cocharacter $\rho : \bG_m \to T$ such that $\bG_m$ acts with strictly positive
weights on $T_pX$ though $\rho$. For example, if $X = G/P$ is a flag variety and
$T \subset G$ is a maximal torus, then all points of $X^T$ are fully definite in
$X$ (see \Section{schubert}). Any non-degenerate $T$-fixed point must be
isolated in $X^T$. Fully definite $T$-fixed points are called \emph{attractive}
in many sources, see e.g.\ \cite{brion:equivariant}; here we follow the
terminology from \cite{bialynicki-birula:some}.

\begin{remark}
  If $X$ is a normal quasi-projective $T$-variety, then $X^{\bG_m} = X^T$ holds
  for all general cocharacters $\rho : \bG_m \to T$. Here a cocharacter is
  called \emph{general} if it avoids finitely many hyperplanes in the lattice of
  all cocharacters. This follows because $X$ admits an equivariant embedding $X
  \subset \bP(V)$, where $V$ is a rational representation of $T$
  \cite{kambayashi:projective, mumford:geometric, sumihiro:equivariant}.
\end{remark}

In the rest of this paper we let $X$ be a non-singular $T$-variety. The
$T$-equivariant Chow cohomology ring of $X$ will be denoted $H^*_T(X)$, see
\cite{fulton:intersection, anderson.fulton:equivariant}. This is an algebra over
the ring $H^*_T(\pt)$, which may be identified with the symmetric algebra of the
character lattice of $T$. Given a class $\sigma \in H_T^*(X)$ and a $T$-fixed
point $p \in X^T$, we let $\sigma_p \in H^*_T(\pt)$ denote the pullback of
$\sigma$ along the inclusion $\{p\} \to X$. When $X$ is defined over $\bK = \C$,
Chow cohomology can be replaced with singular cohomology. In fact, our arguments
will only depend on equivariant classes $[Z]_p \in H_T^*(\pt)$ obtained by
restricting the class of a $T$-stable closed subvariety $Z \subset X$ to a fixed
point, and these restrictions are independent of the chosen cohomology theory.
Similarly, we can use cohomology with coefficients in either $\Z$ or $\Q$.


\section{Equivariant local classes}\label{sec:local}

Let $Z$ be a $T$-variety, fix $p \in Z^T$, and let $\fm \subset \cO_{Z,p}$ be
the maximal ideal in the local ring of $p$. Then the tangent cone $C_pZ =
\Spec(\bigoplus \fm^i/\fm^{i+1})$ is a $T$-stable closed subscheme of the
Zariski tangent space $T_pZ = (\fm/\fm^2)^\vee = \Spec(\Sym(\fm/\fm^2))$. The
\emph{local class} of $Z$ at $p$ is defined by (see
\cite[\S17.4]{anderson.fulton:equivariant})
\begin{equation}
  \eta_p Z = [C_p Z] \ \in\, H^*_T(T_p Z) = H^*_T(\pt) \,.
\end{equation}
When $p$ is a non-singular point of $Z$, we have $\eta_pZ = 1$.

\begin{prop}\label{prop:local}%
  Let $Z$ be a $T$-variety and let $p \in Z^T$ be fully definite in $Z$. Then
  $\eta_p Z \neq 0$ in $H^*_T(\pt)$.
\end{prop}
\begin{proof}
  We may assume that $p$ is a singular point of $Z$, so that $C_pZ$ has positive
  dimension. Choose $\bG_m \subset T$ such that $\bG_m$ acts with positive
  weights on $T_p Z$. It suffices to show that the class of $C_pZ$ is non-zero
  in $H^*_{\bG_m}(T_pZ)$. Let $\{v_1, \dots, v_n\}$ be a basis of $T_pZ$
  consisting of eigenvectors of $\bG_m$. Then the action of $\bG_m$ is given by
  $t.v_i = t^{a_i} v_i$ for positive integers $a_1, \dots, a_n > 0$. Set $A =
  \prod_{i=1}^n a_i$, and let $\bG_m$ act on $U = \bK^n$ by $t.u = t^A u$. Then
  the map $\phi : T_pZ \to U$ defined by
  \[
    \phi(c_1 v_1 + \dots + c_n v_n) = (c_1^{A/a_1}, \dots, c_n^{A/a_n})
  \]
  is a finite $\bG_m$-equivariant morphism. By
  \cite[Thm.~4]{edidin.graham:equivariant} we obtain
  \[
    H^*_{\bG_m}(U \ssm \{0\}) \otimes \Q = H^*(\bP U) \otimes \Q \,,
  \]
  where $\bP U = (U \ssm \{0\})/\bG_m \cong \bP^{n-1}$ is the projective space
  of lines in $U$, and
  \[
    \phi_*[C_p Z]\,|_{U \ssm \{0\}} =
    \deg(\phi)\, [\phi(C_p Z \ssm \{0\})/\bG_m] \ \in H^*(\bP U) \otimes \Q \,.
  \]
  The result now follows from the fact that every non-empty closed subvariety of
  projective space defines a non-zero Chow class.
\end{proof}

\begin{cor}\label{cor:local}%
  Let $X$ be a non-singular $T$-variety, $Z \subset X$ a $T$-stable closed
  subvariety, and $p \in Z^T$ a $T$-fixed point of $Z$. If $p$ is non-degenerate
  in $X$ and fully definite in $Z$, then $[Z]_p \neq 0 \in H^*_T(\pt)$.
\end{cor}
\begin{proof}
  By \cite[Prop.~17.4.1]{anderson.fulton:equivariant} we have $[Z]_p = c_m(T_p
  X / T_p Z) \cdot \eta_p Z$, where $m = \dim T_p X - \dim T_p Z$. The result
  therefore follows from \Proposition{local}, noting that $T$ acts with non-zero
  weights on $T_p X/T_p Z$.
\end{proof}

The following example rules out some potential generalizations of
\Corollary{local}.

\begin{example}
  Let $\bG_m$ act on $\bA^4$ by
  \[
    t.(a,b,c,d) = (ta,\, tb,\, t^{-1}c,\, t^{-1}d) \,.
  \]
  Set $Z = V(ad-bc) \subset \bA^4$, and let $p=(0,0,0,0)$ be the origin in
  $\bA^4$. Then $T_p Z = T_p \bA^4  = \bA^4$ and $C_p Z = Z$. Since $\bG_m$ acts
  trivially on the equation $ad-bc$, we have $\eta_p Z = [Z] = 0$ in
  $H^*_{\bG_m}(\bA^4)$ (see \cite[\S2.3]{anderson.fulton:equivariant}).
\end{example}


\section{Rigidity of convex subvarieties}\label{sec:rigidity}%

Let $T$ be an algebraic torus and let $X$ be a non-singular $T$-variety. We will
show in \Section{schubert} that Schubert varieties and Richardson varieties in a
flag variety $X$ satisfy the following two definitions.

\begin{defn}\label{defn:rigid}%
  A $T$-stable closed subvariety $\Omega \subset X$ is \emph{$T$-equivariantly
  rigid} if it is uniquely determined by its $T$-equivariant cohomology class up
  to a constant. More precisely, if $Z \subset X$ is any $T$-stable closed
  subvariety such that $[Z] = c\,[\Omega]$ holds in $H_T^*(X)$ for some $0 \neq
  c \in \Q$, then $Z = \Omega$.
\end{defn}

\begin{defn}\label{defn:fpi}%
  A $T$-stable closed subvariety $\Omega \subset X$ is \emph{$T$-convex} if, for
  any $T$-stable closed subvariety $Z \subset X$ satisfying $Z^T \subset
  \Omega$, we have $Z \subset \Omega$.
\end{defn}

When the action of $T$ is clear from the context, we frequently drop $T$ from
the notation and write simply \emph{equivariantly rigid} and \emph{convex}. Both
notions are properties of the $T$-equivariant embedding $\Omega \subset X$; for
example, any $T$-variety is convex as a subvariety of itself. Intersections of
$T$-convex subvarieties are again $T$-convex (with the reduced scheme
structure). Most of this paper concerns applications of the following
observation.

\begin{thm}\label{thm:rigid}%
  Let $X$ be a non-singular projective $T$-variety such that all fixed points $p
  \in X^T$ are fully definite in $X$. Then any irreducible $T$-convex subvariety
  of $X$ is $T$-equivariantly rigid.
\end{thm}
\begin{proof}
  Let $\Omega \subset X$ be irreducible and convex, and let $Z \subset X$ be any
  $T$-stable closed subvariety such that $[Z] = c\, [\Omega]$ holds in
  $H_T^*(X)$, with $0 \neq c \in \Q$. Then \Corollary{local} shows that $Z^T =
  \Omega^T = \{ p \in X^T : [Z]_p \neq 0 \}$. Since $\Omega$ is convex, we
  obtain $Z \subset \Omega$. Finally, the assumption $[Z] = c\,[\Omega]$ implies
  that $Z$ and $\Omega$ have the same dimension, so we must have $Z = \Omega$.
\end{proof}

\begin{example}
  Let $X$ be a non-singular projective $T$-variety, let $H^*_T(X)$ be the
  $T$-equivariant Chow cohomology ring, and let $\cL$ be a $T$-equivariant line
  bundle. Given a section $f \in \Gamma(X,\cL)$, the associated divisor $D =
  Z(f)$ is $T$-stable if and only if $f$ is semi-invariant, that is, $f \in
  \Gamma(X,\cL)_\la$ for some character $\la$. In this case $f$ is an
  equivariant section of $\cL \otimes \bK_{-\la}$, hence $[D] = c_1(\cL) -
  c_1(\bK_\la) \in H^*_T(X)$. Moreover, the $T$-stable effective Cartier
  divisors $D'$ satisfying $[D'] = [D]$ are in bijective correspondence with
  $\bP(\Gamma(X,\cL)_\la)$. It follows that if $D$ is reduced and $\dim
  \Gamma(X,\cL^{\otimes m})_{m \la} = 1$ for all $m \in \N$, then $D$ is
  $T$-equivariantly rigid. This observation can be used to produce examples of
  equivariantly rigid subvarieties that are not convex. For example, if $T =
  (\bG_m)^{n+1}$ acts on $\bP^n$ through the standard action on $\bK^{n+1}$,
  then any reduced $T$-stable divisor $D \subset \bP^n$ is equivariantly rigid,
  but $D$ is convex only if it is irreducible, see \Theorem{rigidschub}. We have
  not found an example of an irreducible $T$-stable subvariety that is
  equivariantly rigid but not convex.
\end{example}


\section{Rigidity of Bialynicki-Birula cells}\label{sec:bbcells}%

The multiplicative group $\bG_m$ is identified with the complement of the origin
in $\bA^1$. Given a morphism of varieties $f : \bG_m \to X$, we write
$\lim_{t\to0} f(t) = p$ if $f$ can be extended to a morphism $\Bar{f} : \bA^1
\to X$ such that $\Bar{f}(0) = p$. This limit is unique when it exists, and it
always exists when $X$ is complete.

Let $X$ be a non-singular projective $\bG_m$-variety such that $X^{\bG_m}$ is
finite. Then each fixed point $p \in X^{\bG_m}$ defines the (positive)
Bialynicki-Birula cell
\[
  X_p^+ = \{ x \in X \mid \lim_{t\to0} t.x = p \} \,.
\]
A negative cell is similarly defined by $X_p^- = \{ x \in X \mid \lim_{t\to0}
t^{-1}.x = p\}$. By \cite[Thm.~4.4]{bialynicki-birula:some}, these cells form a
locally closed decomposition of $X$,
\begin{equation}\label{eqn:bbdecomp}%
  X = \bigcup_{p \in X^{\bG_m}} X_p^+ \,,
\end{equation}
that is, a disjoint union of locally closed subsets. In addition, each cell
$X_p^+$ is isomorphic to an affine space.

\begin{lemma}\label{lemma:include}%
  For any $\bG_m$-stable closed subset $Z \subset X$, we have $Z \subset
  {\ds\bigcup_{p \in Z^{\bG_m}} X_p^+}$.
\end{lemma}
\begin{proof}
  For any point $x \in Z$, we have $x \in X_p^+$, where $p = {\ds \lim_{t\to0}
  t.x} \in Z^{\bG_m}$.
\end{proof}

\begin{defn}
  A locally closed decomposition $X = \bigcup X_i$ is called a
  \emph{stratification} if each subset $X_i$ is non-singular and its closure
  $\ov{X_i}$ is a union of subsets $X_j$ of the decomposition.
\end{defn}

The Bialynicki-Birula decomposition \eqn{bbdecomp} typically fails to be a
stratification, for example when $X$ is the blow-up of $\bP^2$ at the point
$[0,1,0]$, where $\bG_m$ acts on $\bP^2$ by $t.[x,y,z] = [x,ty,t^2z]$, see
\cite[Ex.~1]{bialynicki-birula:some*1}. \Lemma{include} shows that the
Bialynicki-Birula decomposition is a stratification if and only if $X_q^+
\subset \ov{X_p^+}$ holds for each fixed point $q \in (\ov{X_p^+})^{\bG_m}$. It
was proved in \cite[Thm.~5]{bialynicki-birula:some*1} that the decomposition is
a stratification when each positive cell $X_p^+$ meets each negative cell
$X_q^-$ transversally. In particular, this holds when $X = G/P$ is a flag
variety and $\bG_m \subset G$ is a general 1-parameter subgroup, see
\cite[Ex.~4.2]{mcgovern:adjoint} or \Lemma{flagvar}. When both the positive and
negative Bialynicki-Birula decompositions are stratifications, all cells $X_p^+$
and $X_q^-$ of complementary dimensions meet transversally, hence the positive
and negative cell closures form a pair of Poincare dual bases of the cohomology
ring $H^*(X)$, see \cite[Lemma~3.11]{benedetti.perrin:cohomology}. In this paper
we utilize the following application, which is a consequence of \Lemma{include}.

\begin{prop}\label{prop:bb-fpi}%
  Assume that the Bialynicki-Birula decomposition of $X$ is a stratification.
  Then each cell closure $\ov{X_p^+} \subset X$ is $\bG_m$-convex.
\end{prop}

\begin{cor}\label{cor:bb-rigid}%
  Let $T$ be an algebraic torus and $X$ a non-singular projective $T$-variety
  such that all fixed points $p \in X^T$ are fully definite in $X$. Assume that
  $X^T = X^{\bG_m}$ for some 1-parameter subgroup $\bG_m \subset T$, such that
  the associated Bialynicki-Birula decomposition of $X$ is a stratification.
  Then each cell closure $\ov{X_p^+}$ is $T$-convex and $T$-equivariantly rigid.
\end{cor}
\begin{proof}
  The cell $X_p^+$ is $T$-stable because $T$ is commutative and $p \in X^T$. The
  result now follows from \Theorem{rigid} and \Proposition{bb-fpi}.
\end{proof}

\begin{question}
  We do not know whether \Proposition{bb-fpi} and \Corollary{bb-rigid} are true
  without the assumption that the Bialynicki-Birula decomposition of $X$ is a
  stratification. It would be very interesting to settle this question.
\end{question}

\begin{example}
  Let $X$ be a non-singular projective toric variety, with torus $T \subset X$,
  and choose $\bG_m \subset T$ such that $X^T = X^{\bG_m}$. We show that the
  conclusion of \Corollary{bb-rigid} holds, even though the Bialynicki-Birula
  decomposition is rarely a stratification. All fixed points $p \in X^T$ are
  fully definite in $X$, as the weights of $T_pX$ form a basis of the character
  lattice of $T$. The $T$-orbits $O_\tau \subset X$ correspond to the cones
  $\tau$ of the fan defining $X$, and we have $O_\sigma \subset \ov{O_\tau}$ if
  and only if $\tau$ is a face of $\sigma$, see
  \cite[\S3.1]{fulton:introduction}. In particular, the $T$-fixed points in $X$
  correspond to the maximal cones $\sigma$. Since $X$ is complete, each cone
  $\tau$ is the intersection of the maximal cones $\sigma$ corresponding to the
  $T$-fixed points in $\ov{O_\tau}$. Since all cell closures $\ov{X_p^+}$ are
  $T$-orbit closures, it suffices to show that each orbit closure $\ov{O_\tau}$
  is $T$-convex. Let $Z \subset X$ be a $T$-stable closed subvariety such that
  $Z^T \subset \ov{O_\tau}$. We may assume that $Z$ is irreducible, in which
  case $Z = \ov{O_\ka}$ is also a $T$-orbit closure. Since $\ka$ is the
  intersection of the maximal cones given by the fixed points in $Z^T$, we
  obtain $\tau \subset \ka$ and $\ov{O_\ka} \subset \ov{O_\tau}$, as required.
  Now assume that $X$ has dimension two. By \cite[Cor.~1 of
  Thm.~4.5]{bialynicki-birula:some}, there is a unique repulsive fixed point $b
  \in X^{\bG_m}$ with $X_b^+ = \{b\}$, and a unique attractive fixed point $a
  \in X^{\bG_m}$ such that $X_a^+$ is a dense open subset of $X$. For all other
  fixed points $p \in X^{\bG_m} \ssm \{a,b\}$, the cell $X_p^+ \cong \bA^1$ is a
  line. If the Bialynicki-Birula decomposition of $X$ is a stratification, then
  $b \in \ov{X_p^+}$ for all $p \in X^{\bG_m}$. The $T$-fixed point $b$
  corresponds to a maximal cone $\sigma$, and $b$ is connected to exactly two
  $T$-stable lines corresponding to the rays forming the boundary of this cone.
  We deduce that $X$ contains at most four $T$-fixed points. Higher dimensional
  toric varieties for which the Bialynicki-Birula decomposition is not a
  stratification can be constructed by taking products. We do not know if the
  cell closures $\ov{X_p^+}$ are $\bG_m$-convex when $X$ is a toric variety.
\end{example}


\section{Rigidity of Richardson varieties}\label{sec:schubert}%

Let $X = G/P = \{g.P \mid g \in G\}$ be a flag variety defined by a connected
reductive linear algebraic group $G$ and a parabolic subgroup $P$. Fix a maximal
torus $T$ and a Borel subgroup $B$ such that $T \subset B \subset P \subset G$.
The opposite Borel subgroup $B^- \subset G$ is defined by $B^- \cap B = T$. Let
$\Phi$ be the root system of non-zero weights of $T_1 G$, the tangent space of
$G$ at the identity element. The positive roots $\Phi^+$ are the non-zero
weights of $T_1 B$. Let $W = N_G(T)/T$ be the Weyl group of $G$, $W_P =
N_P(T)/T$ the Weyl group of $P$, and let $W^P \subset W$ be the subset of
minimal representatives of the cosets in $W/W_P$. The set of $T$-fixed points in
$X$ is given by $X^T \,=\, \{ w.P \mid w \in W \}$, where each point $w.P$
depends only on the coset $w W_P$ in $W/W_P$. Each fixed point $w.P$ defines the
\emph{Schubert varieties} $X_w = \ov{Bw.P}$ and $X^w = \ov{B^-w.P}$. For $w \in
W^P$ we have $\dim(X_w) = \codim(X^w,X) = \ell(w)$. The Bruhat order $\leq$ on
$W^P$ is defined by
\[
  u \leq w \ \ \Leftrightarrow \ \
  X_u \subset X_w \ \ \Leftrightarrow \ \
  X^u \supset X^w \ \ \Leftrightarrow \ \
  X^u \cap X_w \neq \emptyset \,.
\]
A \emph{Richardson variety} is any non-empty intersection $X_w^u = X_w \cap X^u$
of opposite Schubert varieties in $X$. More generally, any $G$-translate of
$X_w^u$ will be called a Richardson variety. Any Richardson variety is reduced,
irreducible, and rational, see \cite{deodhar:some} and
\cite[\S2]{brion.kumar:frobenius}.

Recall that a cocharacter $\rho : \bG_m \to T$ is \emph{strongly dominant} if
$\langle \al, \rho \rangle > 0$ for all positive roots $\al \in \Phi^+$, where
$\langle \al, \rho \rangle \in \Z$ is defined by $\al(\rho(t)) = t^{\langle \al,
\rho \rangle}$ for $t \in \bG_m$. The following lemma is well known, see e.g.\
\cite[Ex.~4.2]{mcgovern:adjoint} or
\cite[Cor.~3.14]{benedetti.perrin:cohomology}.

\begin{lemma}\label{lemma:flagvar}%
  Let $\rho : \bG_m \to T$ be a strongly dominant cocharacter. Then the
  associated Bialynicki-Birula cells of $X$ are given by $X_p^+ = B.p$, for $p
  \in X^T$.
\end{lemma}
\begin{proof}
  Let $\bG_m$ act on $G$ by conjugation through $\rho$. The fixed point set for
  this action is \cite[(7.1.2), (7.6.4)]{springer:linear*1}
  \[
    T = \{ g \in G \mid t g t^{-1} = g ~\forall\, t \in \bG_m \} \,,
  \]
  and the corresponding Bialynicki-Birula cell is
  \cite[(8.2.1)]{springer:linear*1}
  \[
    B = \{ g \in G \mid \lim_{t\to 0} t g t^{-1} \in T \} \,.
  \]
  This implies $B.p \subset X_p^+$ for any fixed point $p \in X^{\bG_m}$. We
  deduce from \eqn{bbdecomp} that the positive Bialynicki-Birula cells in $X$
  are the $B$-orbits.
\end{proof}

\begin{lemma}\label{lemma:stable-trans}%
  Let $Y$ be any $G$-variety, and $\Omega \subset Y$ a $T$-stable closed
  subvariety. Any $T$-stable $G$-translate of $\Omega$ has the form $w.\Omega$,
  with $w \in N_G(T)$.
\end{lemma}
\begin{proof}
  Let $\Omega' = g.\Omega$ be a $T$-stable translate, and let $H \subset G$ be
  the stabilizer of $\Omega'$. Since $T$ and $g T g^{-1}$ are maximal tori in
  $H$, we can choose $h \in H$ such that $T = h g T g^{-1} h^{-1}$. We obtain
  $hg \in N_G(T)$ and $\Omega' = h.\Omega' = hg.\Omega$, as required.
\end{proof}

\begin{thm}\label{thm:rigidschub}%
  Any $T$-stable Richardson variety in the flag variety $X = G/P$ is $T$-convex
  and $T$-equivariantly rigid.
\end{thm}
\begin{proof}
  It follows from \Proposition{bb-fpi} and \Lemma{flagvar} that all Schubert
  varieties $X_w$ and $X^u$ are convex. This implies that every Richardson
  variety $X_w^u = X_w \cap X^u$ is convex, hence all $T$-stable Richardson
  varieties in $X$ are convex by \Lemma{stable-trans}. The $B$-fixed point $p =
  1.P$ is fully definite in $X$ because the weights of $T_p X$ are a subset of
  the negative roots of $G$. Since $W$ acts transitively on $X^T$, this implies
  that all $T$-fixed points in $X$ are fully definite. The result therefore
  follows from \Theorem{rigid}.
\end{proof}

Let $E = G/B$ denote the variety of complete flags, and let $\pi : E \to X$ be
the natural projection. A \emph{projected Richardson variety} in $X$ is the
image $\Pi_w^u(X) = \pi(E_w^u)$ of a Richardson variety in $E$. Projected
Richardson varieties in the Grassmannian $X = \Gr(m,n)$ of type A, obtained as
images of Richardson varieties in $\Fl(n)$, are also called \emph{positroid
varieties}.

\begin{cor}\label{cor:positroid}%
  Let $X = \Gr(m,n)$ be a Grassmannian of type A, and let $T = (\bG_m)^n$ act on
  $X$ through the diagonal action on $\bK^n$. Then all positroid varieties in
  $X$ are $T$-convex and $T$-equivariantly rigid.
\end{cor}
\begin{proof}
  It was proved in \cite{knutson.lam.ea:positroid} that any positroid variety
  $\Omega$ is defined by Plucker equations. Equivalently, $\Omega$ is an
  intersection of $T$-stable Schubert divisors, so $\Omega$ is convex by
  \Theorem{rigidschub} and equivariantly rigid by \Theorem{rigid}.
\end{proof}

\begin{remark}\label{remark:projrich}%
  \Corollary{positroid} does not hold for projected Richardson varieties in
  arbitrary flag varieties $X = G/P$. Each simple root $\be$ defines a projected
  Richardson divisor $D_\be = \Pi_{w_0^P}^{s_\be}(X)$, where $w_0^P$ denotes the
  longest element in $W^P$. It frequently happens that two distinct divisors
  $D_{\be'}$ and $D_{\be''}$ have the same $T$-equivariant cohomology and
  $K$-theory classes, which implies that these divisors are not equivariantly
  rigid. For example, this is the case for the quadric hypersurfaces of
  dimensions $7$ and $8$, of Lie types $B_4$ and $D_5$, and the two-step flag
  variety $\Fl(1,4;5)$ of type $A_4$. For other flag varieties $X$, all
  projected Richardson varieties have distinct equivariant classes, but some
  projected Richardson divisor $D_\be$ contains all $T$-fixed points in $X$,
  which rules out that $D_\be$ is convex. For example, this is the case for the
  Lagrangian Grassmannian $\LG(2,4)$ of type $C_2$ and the maximal orthogonal
  Grassmannian $\OG(4,8)$ of type $D_4$. This is a special case of
  \cite[Lemma~3.1]{benedetti.perrin:cohomology}, which can be used to produce
  many more examples.
\end{remark}

Any element $u \in W$ has a unique factorization $u = u^P u_P$ for which $u^P
\in W^P$ and $u_P \in W_P$, called the \emph{parabolic factorization} with
respect to $P$. This factorization is \emph{reduced} in the sense that $\ell(u)
= \ell(u^P) + \ell(u_P)$. The parabolic factorization of the longest element
$w_0 \in W$ is $w_0 = w_0^P w_{0,P}$, where $w_0^P$ and $w_{0,P}$ are the
longest elements in $W^P$ and $W_P$, respectively. Since $w_0$ and $w_{0,P}$ are
self-inverse, we have $w_{0,P} = w_0 w_0^P$. As preparation for the next
section, we prove the following identity of Schubert varieties.

\begin{lemma}\label{lemma:dualpoint}%
  Let $Q \subset G$ be a parabolic subgroup containing $B$ and set $w = w_0^Q$.
  Then $w^{-1}.X^w = X_{w_0 w}$.
\end{lemma}
\begin{proof}
  Since $X_{w_{0,Q}}$ is a $Q$-stable Schubert variety, we have $X_{w_{0,Q}} =
  w_{0,Q}.X_{w_{0,Q}}$. By translating both sides by $w = w_0^Q$, we obtain
  $w.X_{w_0 w} = w_0.X_{w_0 w} = X^w$.
\end{proof}


\section{Seidel Neighborhoods}\label{sec:seidel}

In this section we prove a conjecture about curve neighborhoods from
\cite{buch.chaput.ea:seidel}. Since this conjecture and its proof relies on the
moduli space of stable maps, we will restrict our attention to varieties defined
over the field $\bK = \C$ of complex numbers. As in \Section{schubert}, we let
$X = G/P$ denote a flag variety.

For any effective degree $d \in H_2(X,\Z)$, we let $M_d = \Mb_{0,3}(X,d)$ denote
the Kontsevich moduli space of 3-pointed stable maps to $X$ of degree $d$ and
genus zero, see \cite{fulton.pandharipande:notes}. The evaluation map $\ev_i :
M_d \to X$, defined for $1 \leq i \leq 3$, sends a stable map to the image of
the $i$-th marked point in its domain. Given two opposite Schubert varieties
$X_v$ and $X^u$, the \emph{Gromov-Witten variety} $M_d(X_v,X^u)$ is the variety
of stable maps that send the first two marked points to $X_v$ and $X^u$:
\[
  M_d(X_v,X^u) \,=\, \ev_1^{-1}(X_v) \cap X_2^{-1}(X^u) \ \subset M_d \,.
\]
The \emph{curve neighborhood} $\Gamma_d(X_v,X^u)$ is the union of all stable
curves of degree $d$ in $X$ connecting $X_v$ and $X^u$:
\[
  \Gamma_d(X_v,X^u) \,=\, \ev_3(M_d(X_v,X^u)) \ \subset X \,.
\]

Let $\Z[q] = \Span_\Z\{ q^d : d \in H_2(X,\Z) \text{ effective} \}$ be the
semigroup ring defined by the effective curve classes on $X$. The equivariant
quantum cohomology ring of $X$ is an algebra over $H_T^*(\pt) \otimes_\Z \Z[q]$,
which is defined by $\QH_T(X) = H^*_T(X) \otimes_\Z \Z[q]$ as a module. The
\emph{quantum product} of two opposite Schubert classes is given by
\[
  [X_v] \star [X^u] \,=\, \sum_{d \geq 0} q^d \ev_{3,*}[M_d(X_v, X^u)] \,,
\]
where the sum is over all effective degrees $d \in H_2(X;\Z)$.

A simple root $\ga \in \Phi^+$ is called \emph{cominuscule} if, when the highest
root is written in the basis of simple roots, the coefficient of $\ga$ is one.
The flag variety $G/Q$ is cominuscule if $Q$ is a maximal parabolic subgroup
corresponding to a cominuscule simple root $\ga$, that is, $s_\ga$ is the unique
simple reflection in $W^Q$. Let $W^\comin \subset W$ be the subset of point
representatives of cominuscule flag varieties of $G$, together with the identity
element:
\[
  W^\comin \,=\, \{ w_0^Q \mid G/Q \text{ is cominuscule} \} \cup \{1\} \,.
\]
This is a subgroup of $W$, which is isomorphic to the quotient of the coweight
lattice of $\Phi$ modulo the coroot lattice
\cite[Prop.~VI.2.6]{bourbaki:elements*78}. The isomorphism sends $w_0^Q$ to the
class of the fundamental coweight $\om_\ga^\vee$ corresponding to $Q$. Notice
that $\ga$ is the unique simple root for which $w_0^Q.\ga < 0$. In the following
we set $d(w_0^Q, u) = \omega_\ga^\vee - u^{-1}.\omega_\ga^\vee \in H_2(X;\Z)$
for any $u \in W$. Here we identify the group $H_2(X,\Z)$ with a quotient of the
coroot lattice, by mapping each simple coroot $\be^\vee$ to the curve class
$[X_{s_\be}]$ if $s_\be \in W^P$, and to zero otherwise.

The \emph{Seidel representation} of $W^\comin$ on $\QH(X)/\langle q-1 \rangle$
is defined by $w.[X^u] = [X^w] \star [X^u]$ for $w \in W^\comin$ and $u \in W$.
In fact, we have \cite{seidel:1, belkale:transformation,
chaput.manivel.ea:affine}
\begin{equation}\label{eqn:seidel}%
  [X^w] \star [X^u] \,=\, q^{d(w,u)}\, [X^{w u}]
\end{equation}
in the (non-equivariant) quantum ring $\QH(X)$. This implies that $d(w,u)$ is
the unique minimal degree $d$ for which $\Gamma_d(X_{w_0 w}, X^u)$ is not empty
\cite{fulton.woodward:quantum, buch.chung.ea:euler}. More generally, it was
proved in \cite{chaput.manivel.ea:affine, chaput.perrin:affine} that the
identity
\begin{equation}\label{eqn:htseidel}%
  [X^w] \star [w.X^u] \,=\, q^{d(w,u)}\, [X^{w u}]
\end{equation}
holds in the equivariant quantum cohomology ring $\QH_T(X)$. We will discuss
generalizations to quantum $K$-theory in \Section{qkseidel}.

It follows from \eqn{seidel} and the definition of the quantum product in
$\QH(X)$ that $[\Gamma_{d(w,u)}(X_{w_0w}, X^u)] = [X^{w u}]$ holds in $H^*(X)$.
Conjecture~3.11 from \cite{buch.chaput.ea:seidel} asserts that
$\Gamma_{d(w,u)}(X_{w_0 w}, X^u)$ is in fact equal to the translated Schubert
variety $w^{-1}.X^{w u}$. This is proved below as a consequence of
\Theorem{rigidschub} and \eqn{htseidel}. This result was known when $X = G/P$ is
cominuscule and $w = w_0^P$ \cite{buch.chaput.ea:seidel}, when $X$ is a
Grassmannian of type A and $[X^w]$ is a special Seidel class
\cite[Cor.~4.6]{li.liu.ea:seidel}, when $X$ is any flag variety of type A
\cite{tarigradschi:curve}, and when $X$ is the symplectic Grassmannian
$\SG(2,2n)$ \cite[Thm.~8.1]{benedetti.perrin.ea:quantum}.

\begin{thm}\label{thm:seidelnbhd}%
  Let $X = G/P$ be a complex flag variety. For $w \in W^\comin$ and $u \in W$ we
  have $\Gamma_{d(w,u)}(X_{w_0 w}, X^u) = w^{-1}.X^{w u}$.
\end{thm}
\begin{proof}
  By applying $w^{-1}$ to both sides of \eqn{htseidel} and using
  \Lemma{dualpoint}, we obtain
  \[
    [X_{w_0 w}] \star [X^u] \,=\, q^{d(w,u)}\, [w^{-1}.X^{w u}]
  \]
  in $\QH_T(X)$. By definition of the quantum product, this implies that
  \[
    [w^{-1}.X^{w u}] \,=\, \ev_{3,*} [M_{d(w,u)}(X_{w_0 w}, X^u)]
    \,=\, c\, [\Gamma_{d(w,u)}(X_{w_0 w}, X^u)]
  \]
  holds in $H^*_T(X)$, where $c$ is the degree of the map $\ev_3 :
  M_{d(w,u)}(X_{w_0 w}, X^u) \to \Gamma_{d(w,u)}(X_{w_0 w}, X^u)$. The result
  therefore follows from \Theorem{rigidschub}.
\end{proof}


\section{Seidel products in quantum $K$-theory}\label{sec:qkseidel}%

In this section we discuss a generalization of the Seidel multiplication formula
to quantum $K$-theory. We start by briefly recalling the definition of quantum
$K$-theory. A more detailed discussion can be found in
\cite[\S2]{buch.chaput.ea:chevalley}.

Let $X = G/P$ be a flag variety defined over $\bK = \C$. The equivariant
$K$-theory ring $K^T(X)$ is an algebra over the representation ring $\Gamma =
K^T(\pt)$. The equivariant quantum $K$-theory ring $\QK_T(X)$ was originally
constructed by Givental and Lee \cite{givental:wdvv, lee:quantum*1}. This ring
is an algebra over the formal power series ring $\Gamma\llbracket q \rrbracket =
\Gamma\llbracket q_\be : s_\be \in W^P \rrbracket$, which has one variable
$q_\be$ for each simple reflection $s_\be$ in $W^P$. As a module over
$\Gamma\llbracket q \rrbracket$ we have $\QK_T(X) = K^T(X) \otimes_\Gamma
\Gamma\llbracket q \rrbracket$. The \emph{undeformed product} of two opposite
Schubert classes in $\QK_T(X)$ is defined by
\[
  [\cO_{X_v}] \odot [\cO_{X^u}] \,=\,
  \sum_{d \geq 0} \,q^d \ev_{3,*} [\cO_{M_d(X_v,X^u)}] \,.
\]
Let $\Psi : \QK_T(X) \to \QK_T(X)$ be the $\Gamma\llbracket q \rrbracket$-linear
map defined by
\[
  \Psi([\cO_{X^w}]) \,=\, \sum_{d \geq 0} \,q^d\, [\cO_{\Gamma_d(X^w)}] \,,
\]
where the curve neighborhood $\Gamma_d(X^w) = \ev_2(\ev_1^{-1}(X^w))$ is defined
using the evaluation maps from $M_d$. This curve neighborhood is a Schubert
variety in $X$ by \cite[Prop.~3.2(b)]{buch.chaput.ea:finiteness}, whose Weyl
group element was determined in \cite{buch.mihalcea:curve}. By
\cite[Prop.~2.3]{buch.chaput.ea:chevalley}, Givental's \emph{quantum $K$-theory
product} $\star$ is given by
\begin{equation}\label{eqn:qkproduct}%
  [\cO_{X_v}] \star [\cO_{X^u}] \,=\,
  \Psi^{-1}([\cO_{X_v}] \odot [\cO_{X^u}]) \,.
\end{equation}
The following conjecture is the $K$-theoretic analogue of the Seidel
multiplication formula \eqn{htseidel} in $\QH_T(X)$ proved in
\cite{chaput.manivel.ea:affine, chaput.perrin:affine}.

\begin{conj}\label{conj:qkseidel}%
  For $w \in W^\comin$ and $u \in W$ we have
  \[
    [\cO_{X_{w_0 w}}] \star [\cO_{X^u}] \,=\,
    q^{d(w,u)}\, [\cO_{w^{-1}.X^{w u}}]
    \text{ \ \ and \ \ }
    [\cO_{X^w}] \star [\cO_{w.X^u}] \,=\,
    q^{d(w,u)}\, [\cO_{X^{w u}}]
  \]
  in $\QK_T(X)$.
\end{conj}

The two identities in \Conjecture{qkseidel} are equivalent by \Lemma{dualpoint}.
The non-equivariant case of this conjecture was proved in
\cite[Cor.~3.7]{buch.chaput.ea:seidel} when $X$ is a cominuscule flag variety.
We will extend this result to equivariant quantum $K$-theory below, based on the
following conjectural generalization of \Theorem{seidelnbhd}. Recall that a
morphism $\pi : Z \to Y$ is called \emph{cohomologically trivial} if $\pi_*
\cO_Z = \cO_Y$ and $R^j \pi_* \cO_Z = 0$ for $j \geq 1$.

\begin{conj}\label{conj:seidelnbhd}%
  Let $w \in W^\comin$, $u \in W$, and let $e \in H_2(X,\Z)$ be
  effective.\smallskip

  \noin{\rm(a)} We have $\Gamma_{d(w,u)+e}(X_{w_0 w}, X^u) =
  \Gamma_e(w^{-1}.X^{w u})$.\smallskip

  \noin{\rm(b)} The evaluation map $\ev_3 : M_{d(w,u)+e}(X_{w_0 w}, X^u) \to
  \Gamma_{d(w,u)+e}(X_{w_0 w}, X^u)$ is cohomologically trivial.
\end{conj}

\Conjecture{seidelnbhd} is a variant of the quantum-equals-classical theorem for
Gromov-Witten invariants as stated in \cite[Thm.~4.1]{buch.chaput.ea:projected},
see also \cite[Thm.~1.2]{xu:quantum}. The conjecture is true for $e=0$; part (a)
is equivalent to \Theorem{seidelnbhd}, and part (b) holds because the map $\ev_3
: M_{d(w,u)}(X_{w_0 w}, X^u) \to \Gamma_{d(w,u)}(X_{w_0 w}, X^u)$ is birational
by \cite{belkale:transformation, chaput.manivel.ea:affine}, and
$M_{d(w,u)}(X_{w_0 w}, X^u)$ has rational singularities by
\cite[Cor.~3.1]{buch.chaput.ea:finiteness}. For $e \geq 0$, \Theorem{seidelnbhd}
implies that
\begin{equation}\label{eqn:seidelnbhd_incl}%
  \Gamma_e(w^{-1}.X^{w u}) \,=\, \Gamma_e(\Gamma_{d(w,u)}(X_{w_0 w}, X^u))
  \,\subset\, \Gamma_{d(w,u)+e}(X_{w_0 w}, X^u) \,,
\end{equation}
and $\Gamma_{d(w,u)+e}(X_{w_0 w}, X^u)$ is irreducible by
\cite[Cor.~3.8]{buch.chaput.ea:finiteness}. \Conjecture{seidelnbhd}(a) is
therefore true if and only if $\Gamma_{d(w,u)+e}(X_{w_0 w}, X^u)$ and
$\Gamma_e(X^{w u})$ have the same dimension. We prove below that
\Conjecture{seidelnbhd}(a) is true when $X = \GL(n)/P$ is any flag variety of
Lie type A. \Conjecture{qkseidel} follows from \Conjecture{seidelnbhd} by the
following observation.

\begin{lemma}\label{lemma:qkseidel_equiv}%
  Given $w \in W^\comin$ and $u \in W$, the identity $[\cO_{X_{w_0 w}}] \star
  [\cO_{X^u}] = q^{d(w,u)}\, [\cO_{w^{-1}.X^{w u}}]$ holds in $\QK_T(X)$ if and
  only if
  \begin{equation}\label{eqn:seidelpush}%
    \ev_{3,*}[\cO_{M_{d(w,u)+e}(X_{w_0 w}, X^u)}] \,=\,
    [\cO_{\Gamma_e(w^{-1}.X^{w u})}]
  \end{equation}
  holds in $K_T(X)$ for all effective degrees $e \in H_2(X,\Z)$.
\end{lemma}
\begin{proof}
  Both assertions are equivalent to the identity
  \[
    [\cO_{X_{w_0 w}}] \odot [\cO_{X^u}] \,=\,
    \sum_{e \geq 0} q^{d(w,u)+e}\, [\cO_{\Gamma_e(w^{-1}.X^{w u})}]
  \]
  by the definition \eqn{qkproduct} of the quantum product in $\QK_T(X)$.
\end{proof}

\begin{thm}\label{thm:qkseidel_comin}%
  \Conjecture{qkseidel} and \Conjecture{seidelnbhd} are true when $X$ is a
  cominuscule flag variety.
\end{thm}
\begin{proof}
  Assume that $X$ is cominuscule. Then \Conjecture{seidelnbhd}(b) is a special
  case of \cite[Thm.~4.1]{buch.chaput.ea:projected}, and
  \Conjecture{seidelnbhd}(a) follows from \Theorem{seidelnbhd} and
  \cite[Cor.~8.24]{buch.chaput.ea:positivity}, noting that $q^{d(w,u)}$ is the
  maximal power of $q$ occurring in the quantum cohomology product $[X_{w_0 w}]
  \star [X^u]$ by \cite{belkale:transformation, chaput.manivel.ea:affine}. This
  proves \Conjecture{seidelnbhd}, which implies \Conjecture{qkseidel} by
  \Lemma{qkseidel_equiv}.
\end{proof}

We finish this section by proving that \Conjecture{seidelnbhd}(a) can be reduced
to the case where $X$ is a flag variety of Picard rank 1. In particular,
\Conjecture{seidelnbhd}(a) follows from \Theorem{qkseidel_comin} in type A.
These results were proved for $e=0$ in \cite{tarigradschi:curve}. We thank
Mihail Tarigradschi for suggesting that his methods might apply to the general
case of our conjecture.

Recall that $X = G/P$. Let $Q_1, Q_2 \subset G$ be parabolic subgroups such that
$P = Q_1 \cap Q_2$. Set $Y_i = G/Q_i$ and let $\pi_i : X \to Y_i$ be the
projection, for $i \in \{1,2\}$. Given a degree $d \in H_2(X,\Z)$, we also let
$d$ denote the image $\pi_{i,*}(d)$ of this degree in $H_2(Y_i,\Z)$. Let
$\Gamma_d(Y_{i,v},Y_i^u) \subset Y_i$ be the union of all stable curves of
degree $d$ in $Y_i$ that connect the Schubert varieties $Y_{i,v} = \pi_i(X_v)$
and $Y_i^u = \pi_i(X^u)$, for $u,v \in W$. The following result generalizes
\cite[Thm.~2.6.1]{bjorner.brenti:combinatorics} and
\cite[Lemma~4]{tarigradschi:curve}.

\begin{lemma}\label{lemma:nbhd_proj}%
  We have $\Gamma_d(X^u) = \pi_1^{-1}(\Gamma_d(Y_1^u)) \cap
  \pi_2^{-1}(\Gamma_d(Y_2^u))$.
\end{lemma}
\begin{proof}
  Let $\dist_X(X_v,X^u)$ denote the unique minimal degree of a rational curve in
  $X$ connecting $X_v$ and $X^u$. It follows from
  \cite[Thm.~5]{buch.chung.ea:euler} that this degree is uniquely determined by
  $\pi_{i,*}(\dist_X(X_v,X^u)) = \dist_{Y_i}(Y_{i,v},Y_i^u)$ for $i \in
  \{1,2\}$. Using that $v.P \in \Gamma_d(X^u)$ holds if and only if $d \geq
  \dist_X(X_v,X^u)$, we deduce that $\Gamma_d(X^u)$ and
  $\pi_1^{-1}(\Gamma_d(Y_1^u)) \cap \pi_2^{-1}(\Gamma_d(Y_2^u))$ contain the
  same $T$-fixed points. The lemma follows from this, as both sets are
  $B^-$-stable subvarieties of $X$.
\end{proof}

The following result implies that \Conjecture{seidelnbhd}(a) follows from the
case where $X$ has Picard rank 1. It was proved for $e=0$ in
\cite[Thm.~3]{tarigradschi:curve}.

\begin{thm}\label{thm:seidelnbhd_reduce}%
  Let $X = G/P$, $Y_1 = G/Q_1$, and $Y_2 = G/Q_2$ be flag varieties such that $P
  = Q_1 \cap Q_2$. Let $w \in W^\comin$, $u \in W$, and let $e \in H_2(X,\Z)$ be
  any effective degree. If $\Gamma_{d(w,u)+e}(Y_{i,w_0 w},Y_i^u) =
  \Gamma_e(w^{-1}.Y_i^{w u})$ holds for $i \in \{1,2\}$, then
  $\Gamma_{d(w,u)+e}(X_{w_0 w},X^u) = \Gamma_e(w^{-1}.X^{w u})$.
\end{thm}
\begin{proof}
  The assumptions and \Lemma{nbhd_proj} imply that
  \[
    \begin{split}
      \Gamma_{d(w,u)+e}(X_{w_0 w},X^u)
      &\subset
      \pi_1^{-1}(\Gamma_{d(w,u)+e}(Y_{1,w_0 w},Y_1^u)) \cap
      \pi_2^{-1}(\Gamma_{d(w,u)+e}(Y_{2,w_0 w},Y_2^u)) \\
      &=
      \pi_1^{-1}(\Gamma_e(w^{-1}.Y_1^{w u})) \cap
      \pi_2^{-1}(\Gamma_e(w^{-1}.Y_2^{w u}))
      =
      \Gamma_e(w^{-1}.X^{w u}) \,,
    \end{split}
  \]
  and the opposite inclusion holds by \eqn{seidelnbhd_incl}.
\end{proof}

\begin{cor}
  \Conjecture{seidelnbhd}(a) is true when $X = \GL(n)/P$ has Lie type A.
\end{cor}
\begin{proof}
  This follows from \Theorem{qkseidel_comin} and \Theorem{seidelnbhd_reduce},
  noting that all flag varieties of type A with Picard rank 1 are Grassmannians,
  and therefore cominuscule.
\end{proof}


\section{Horospherical varieties of Picard rank 1}\label{sec:horospherical}%

In this section we interpret \Theorem{rigid} and \Proposition{bb-fpi} for a
class of horospherical varieties that includes all non-singular projective
horospherical varieties of Picard rank 1 (except flag varieties) by Pasquier's
classification \cite{pasquier:some}. Let $G$ be a connected reductive linear
algebraic group, $B \subset G$ a Borel subgroup, and $T \subset B$ a maximal
torus. Let $V_1$ and $V_2$ be irreducible rational representations of $G$, and
let $v_i \in V_i$ be a highest weight vector of weight $\la_i$, for $i \in
\{1,2\}$. We assume that $\la_1 \neq \la_2$. Define
\[
  X = \ov{G.[v_1+v_2]} \subset \bP(V_1\oplus V_2) \,.
\]
If $X$ is normal, then $X$ is a horospherical variety of rank 1, see
\cite[Ch.~7]{timashev:homogeneous}. We will assume that $X$ is non-singular and
$\bK = \C$, even though many claims hold more generally; this implies that $X$
is fibered over a flag variety $G/P_{12}$ with non-singular horospherical fibers
of Picard rank 1, see \Remark{pasfib}. Any $G$-translate of a $B$-orbit closure
in $X$ will be called a \emph{Schubert variety}. Our next result uses the action
of $T \times \bG_m$ on $X$ defined by $(t,z).[u_1+u_2] = t.[u_1 + z u_2]$, for
$u_i \in V_i$. We have $X^{T \times \bG_m} = X^T$, and a Schubert variety is
$T$-stable if and only if it is $T \times \bG_m$-stable.

\begin{thm}\label{thm:horo}%
  Any $T$-stable Schubert variety in $X$ is $T \times \bG_m$-convex and $T
  \times \bG_m$-equivariantly rigid.
\end{thm}

Before proving \Theorem{horo}, we sketch elementary proofs of some basic facts
about $X$, which are also consequences of general results about spherical
varieties, see \cite{timashev:homogeneous, perrin:geometry, pasquier:some} and
the references therein.

Given an element $[u_1+u_2] \in \bP(V_1\oplus V_2)$, we will always assume $u_i
\in V_i$, and $i$ will always mean an element from $\{1,2\}$. We consider
$\bP(V_i)$ as a subvariety of $\bP(V_1\oplus V_2)$. Let $\pi_i : \bP(V_1\oplus
V_2) \ssm \bP(V_{3-i}) \to \bP(V_i)$ denote the projection from $V_{3-i}$,
defined by $\pi_i([u_1+u_2]) = [u_i]$. Set $X_0 = G.[v_1+v_2] \subset
\bP(V_1\oplus V_2)$, $X_i = G.[v_i] \subset \bP(V_i)$, and $X_{12} =
G.([v_1],[v_2]) \subset \bP(V_1) \times \bP(V_2)$. Since $v_i$ is a highest
weight vector, the stabilizer $P_i = G_{[v_i]}$ is a parabolic subgroup
containing $B$. It follows that $X_i \cong G/P_i$ and $X_{12} \cong G/(P_1\cap
P_2)$ are flag varieties. In particular, $X_i$ is closed in $\bP(V_i)$, and
$X_{12}$ is closed in $\bP(V_1)\times \bP(V_2)$. Notice also that $X_0 \cong
G/H$, where $H \subset P_1 \cap P_2$ is the kernel of the character $\la_1-\la_2
: P_1\cap P_2 \to \bG_m$. This shows that $X_0$ is a $\bG_m$-bundle over
$X_{12}$, so $X$ is a non-singular projective horospherical variety of rank 1
(but not necessarily of Picard rank 1, see \Remark{pasfib}).

Let $W$ be the Weyl group of $G$, and recall the notation from
\Section{schubert}.

\begin{lemma}\label{lemma:orbits}%
  We have $X = X_0 \cup X_1 \cup X_2$. The $B$-orbit closures in $X$ are
  \[
    \begin{split}
      \ov{B w.[v_i]} \, &= \bigcup_{w' \leq w} B w'.[v_i]
      \text{ \ \ for $w \in W^{P_i}$ and $i \in \{1,2\}$, and} \\
      \ov{B w.[v_1+v_2]} \, &=
      \bigcup_{w' \leq w} \left(
        B w'.[v_1+v_2] \cup B w'.[v_1] \cup B w'.[v_2] \right)
      \text{\ \ for $w \in W^{P_1 \cap P_2}$.}
    \end{split}
  \]
\end{lemma}
\begin{proof}
  Set $\bP_0 = \bP(V_1\oplus V_2) \ssm (\bP(V_1) \cup \bP(V_2))$. Since $\la_1
  \neq \la_2$, it follows that $\ov{T.[v_1+v_2]}$ is the line through $[v_1]$
  and $[v_2]$ in $\bP(V_1\oplus V_2)$. This implies $X_0 = (\pi_1 \times
  \pi_2)^{-1}(X_{12})$, hence $X_0$ is closed in $\bP_0$, and $X_0 = X \cap
  \bP_0$. We also have $X_i \subset X \cap \bP(V_i) \subset \pi_i^{-1}(X_i) \cap
  \bP(V_i) = X_i$, which proves the first claim. To finish the proof, it
  suffices to show $w'.[v_i] \in \ov{Bw.[v_1+v_2]}$ if and only if $w' \leq w$
  (when $w' \in W^{P_i})$. The implication `if' holds because $w'.[v_i] \in
  \ov{Tw'.[v_1+v_2]}$, and `only if' holds because $\pi_i(\ov{Bw.[v_1+v_2]} \ssm
  X_{3-i}) \subset \ov{Bw.[v_i]}$.
\end{proof}

Define an alternative action of $P_i$ on $V_{3-i}$ by $p \bullet u =
\la_i(p)^{-1} p.u$, and use this action to form the space
\[
  G \times^{P_i} V_{3-i} \ = \ \{[g,u] : g \in G, u \in V_{3-i}\} \ / \
  \{[g p,u] = [g, p \bullet u] : p \in P_i \} \,.
\]
Define a morphism of varieties $\phi_i : G \times^{P_i} V_{3-i} \to
\bP(V_1\oplus V_2)$ by $\phi_i([g,u]) = g.[v_i+u]$. This is well defined since
$p.(v_i+u) = \la_i(p) (v_i + p \bullet u)$ holds for $p \in P_i$ and $u \in
V_{3-i}$. Set $E_i = (P_i \bullet v_{3-i}) \cup \{0\} \subset V_{3-i}$. Noting
that $E_i$ is the cone over $P_i.[v_{3-i}] \cong P_i/(P_1 \cap P_2)$, it follows
that $E_i$ is closed in $V_{3-i}$.

\begin{lemma}\label{lemma:vb}%
  The restricted map $\phi_i : G \times^{P_i} E_i \to X_0 \cup X_i$ is an
  isomorphism of varieties. In particular, $E_i \subset V_{3-i}$ is a linear
  subspace.
\end{lemma}
\begin{proof}
  Assume $\phi_i([g,u]) = \phi_i([g',u'])$, and set $p = g^{-1} g'$. We obtain
  $p \in P_i$ and $[v_i + u] = p.[v_i + u'] = [v_i + p\bullet u']$ in
  $\bP(V_1\oplus V_2)$, hence $[g,u] = [g,p\bullet u'] = [g p, u'] = [g', u']$
  in $G \times^{P_i} V_{3-i}$. We deduce that $\phi_i : G \times^{P_i} E_i \to
  X_0 \cup X_i$ is bijective, so the lemma follows from Zariski's main theorem,
  using that $X_0 \cup X_i$ is non-singular.
\end{proof}

Fix a strongly dominant cocharacter $\rho : \bG_m \to T$. For $a \in \Z$, define
$\rho_a : \bG_m \to T \times \bG_m$ by $\rho_a(z) = (\rho(z), z^a)$. The
resulting action of $\bG_m$ on $X$ is given by $\rho_a(z).[u_1+u_2] =
\rho(z).[u_1 + z^a u_2]$.

\begin{lemma}\label{lemma:horo_definite}%
  All $T$-fixed points in $X$ are fully definite for the action of
  $T \times \bG_m$.
\end{lemma}
\begin{proof}
  \Lemma{vb} shows that $[v_1]$ has a $T\times \bG_m$-stable open neighborhood
  in $X$ isomorphic to $B^-.[v_1] \times E_1$, where the action is given by
  $(t,z).(x,u) = (t.x, t\bullet z u)$. If $a$ is sufficiently negative, then
  $\bG_m$ acts through $\rho_a$ on $T_{[v_1]}X = T_{[v_1]}X_1 \oplus E_1$ with
  strictly negative weights, hence $[v_1]$ is fully definite in $X$ for the
  action of $T \times \bG_m$. A symmetric argument shows that $[v_2]$ is fully
  definite. The result follows from this, since all $T$-fixed points in $X$ are
  obtained from $[v_1]$ or $[v_2]$ by the action of the Weyl group $W$.
\end{proof}

\begin{proof}[Proof of \Theorem{horo}]
  For $a$ sufficiently negative, it follows from \Lemma{flagvar} that the
  Bialynicki-Birula cells of $X$ defined by $\rho_a$ are
  \[
    X^+_{w.[v_1]} = B w.[v_1]
    \text{ \ \ \ and \ \ \ }
    X^+_{w.[v_2]} = B w.[v_1+v_2] \cup B w.[v_2] \,.
  \]
  These cells form a stratification of $X$ by \Lemma{orbits}, so
  \Proposition{bb-fpi} implies that $\ov{B w.[v_1]}$ and $\ov{B w.[v_1+v_2]}$
  are $T \times \bG_m$-convex for $w \in W$. A symmetric argument applies to
  $\ov{B w.[v_2]}$, hence all $T$-stable Schubert varieties in $X$ are $T \times
  \bG_m$-convex by \Lemma{stable-trans}. The result now follows from
  \Theorem{rigid} and \Lemma{horo_definite}.
\end{proof}

\begin{remark}\label{remark:pasfib}%
  The exact sequence of \cite[Thm.~3.2.4]{perrin:geometry} implies that
  $\Pic(X)$ is a free abelian group of rank equal to the rank of $X$ (which is
  one) plus the number of $B$-stable prime divisors in $X$ that do not contain a
  $G$-orbit. Any $B$-stable prime divisor meeting $X_0$ has the form $D = \ov{B
  w_0 s_\be.[v_1+v_2]}$, where $\be$ is a simple root, and \Lemma{orbits} shows
  that $D$ contains $X_i$ if and only if $\be$ is a root of $P_i$. Let $P_{12}
  \subset G$ be the parabolic subgroup generated by $P_1$ and $P_2$. We obtain
  $\Pic(X) \cong \Z \oplus \Pic(G/P_{12})$. Let $\pi : X \to G/P_{12}$ be the
  map defined by $\pi(g.[v_1+v_2]) = \pi(g.[v_i]) = g.P_{12}$. This is a
  $G$-equivariant morphism of varieties, as its restriction to $X_0 \cup X_i$ is
  the composition of $\pi_i : X_0\cup X_i \to G/P_i$ with the projection $G/P_i
  \to G/P_{12}$. The fibers of $\pi$ are translates of $\pi^{-1}(1.P_{12}) =
  \ov{L.[v_1+v_2]} \subset \bP(V_1\oplus V_2)$, where $L$ is the Levi subgroup
  of $P_{12}$ containing $T$. Moreover, $\pi^{-1}(1.P_{12})$ is a non-singular
  projective horospherical variety of Picard rank 1, so it is either a flag
  variety or one of the non-homogeneous spaces from Pasquier's classification
  \cite{pasquier:some}.
\end{remark}

\begin{question}
  Let $X$ be any projective $G$-horospherical variety fibered over a flag
  variety $G/P$ with non-singular horospherical fibers of Picard rank 1. Is it
  true that $X$ is isomorphic to an orbit closure $\ov{G.[v_1+v_2]} \subset
  \bP(V)$, where $V$ is a rational representation of $G$, and $v_1, v_2 \in V$
  are highest weight vectors?
\end{question}

\begin{example}
  Let $X$ be the blow-up of $\bP^2$ at a point $p$, let $\pi : X \to \bP^1$ be
  the morphism defined by projection from $p$, and set $G = \SL(2,\C)$. Then $X$
  is $G$-horospherical and fibered over $\bP^1$ with fiber $\bP^1$. This variety
  $X$ is isomorphic to $\ov{G.[v_1+v_2]} \subset \bP(V_1 \oplus V_2)$, where
  $v_1$ is a highest weight vector in $V_1 = \C^2$, and $v_2$ is a highest
  weight vector in $V_2 = \Sym^2(\C^2)$.
\end{example}

\ifdefined\mybibfile
\bibliography{\mybibfile}
\else

\fi
\bibliographystyle{halpha}

\end{document}